\documentclass[preprint,12pt]{elsarticle}
\usepackage[utf8]{inputenc}

\usepackage{amsthm,amsmath,amsfonts,amssymb}
\usepackage[english]{babel}
\usepackage{graphicx}
\usepackage{verbatim}

\DeclareMathOperator{\dist}{dist}

\def\rep{\mathop{\rm rep }\nolimits}
\def\Rep{\mathop{\rm Rep }\nolimits}

\DeclareMathOperator{\spec}{sp}

\newcommand\blfootnote[1]{%
  \begingroup
  \renewcommand\thefootnote{}\footnote{#1}%
  \addtocounter{footnote}{-1}%
  \endgroup
}

\def\z{\mbox{\boldmath $z$}}

\def\vec0{\mbox{\boldmath $0$}}

\def\I{\mbox{\boldmath $I$}}

\def\M{\mbox{\boldmath $M$}}

\def\I{\mbox{\boldmath $I$}}

\def\M{\mbox{\boldmath $M$}}

\theoremstyle{plain}   

\newtheorem{theorem}{Theorem}[section]
\newtheorem{proposition}[theorem]{Proposition}
\newtheorem{corollary}[theorem]{Corollary}

\begin{document}

\begin{frontmatter}

\title{An improved upper bound for the order of mixed graphs}


\author{C. Dalf\'o}
\address{Dept. de Matem\`atiques, Universitat Polit\`ecnica de Catalunya, Barcelona, Catalonia}
\ead{cristina.dalfo@upc.edu}

\author{M. A. Fiol}
\address{Dept. de Matem\`atiques, Barcelona Graduate School of Mathematics, Universitat Polit\`ecnica de Catalunya, Barcelona, Catalonia}
\ead{miguel.angel.fiol@upc.edu}

\author{N. L\'opez}
\address{Dept. de Matem\`atica, Universitat de Lleida, Lleida, Spain}
\ead{nlopez@matematica.udl.es}

\begin{abstract}
A  mixed graph $G$ can contain both (undirected) edges and arcs (directed edges).
Here we derive an improved Moore-like bound for the maximum number of vertices of a mixed graph with diameter at least three. Moreover, a complete enumeration of all optimal $(1,1)$-regular mixed graphs with diameter three is presented, so proving that, in general, the proposed bound cannot be improved.
\end{abstract}


\end{frontmatter}

\blfootnote{
\begin{minipage}[l]{0.3\textwidth} \includegraphics[trim=10cm 6cm 10cm 5cm,clip,scale=0.15]{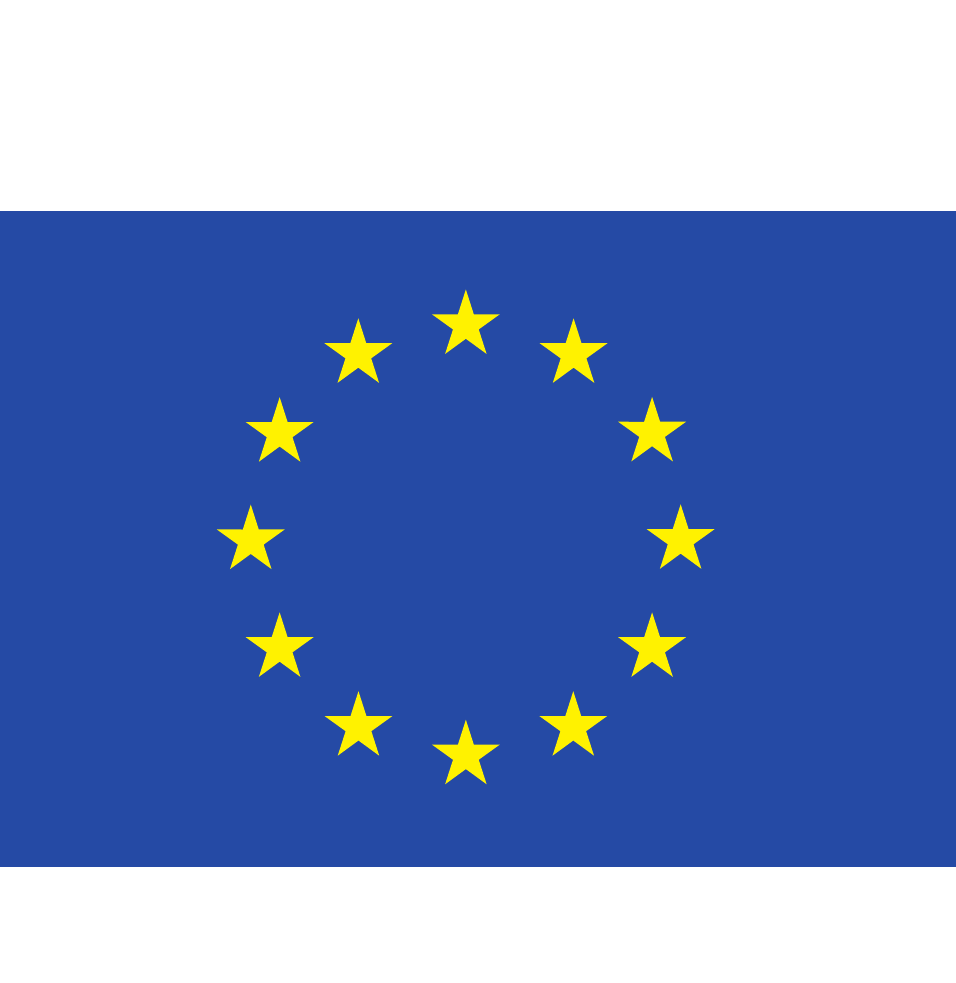} \end{minipage}  \hspace{-2cm} \begin{minipage}[l][1cm]{0.79\textwidth}
   The research of C. Dalf\'o has also received funding from the European Union's Horizon 2020 research and innovation programme under the Marie Sk\l{}odowska-Curie grant agreement No 734922.
  \end{minipage}}

\section{Introduction}
A {\em mixed} (or {\em partially directed\/}) graph $G=(V,E,A)$ consists of a set $V$ of vertices, a set $E$ of edges, or unordered pairs of vertices, and a set $A$ of arcs, or ordered pairs of vertices.
Thus, $G$ can also be seen as a digraph having {\em digons\/}, or pairs of opposite arcs between some pairs of vertices.
If there is an edge between vertices $u,v\in V$, we denote it by $u\sim v$, whereas if there is an arc from $u$ to $v$, we write $u\rightarrow v$.
We denote by $r(u)$ the {\em undirected degree} of $u$, or the number of edges incident to $u$. Moreover, the {\em out-degree\/} [respectively, {\em in-degree\/}] of  $u$, denoted by $z^+(u)$ [respectively, $z^-(u)$], is the number of arcs emanating from [respectively, to] $u$. If $z^+(u)=z^-(u)=z$ and $r(u)=r$, for all $u \in V$, then $G$ is said to be {\em totally regular\/} of degrees $(r,z)$, with $r+z=d$ (or simply {\em $(r,z)$-regular\/}).
The length of a shortest path from $u$ to $v$ is the {\it distance\/} from $u$ to $v$, and it is denoted by $\dist(u,v)$. Note that $\dist(u,v)$ may be different from $\dist(v,u)$ when the shortest paths between $u$ and $v$ involve arcs. The maximum distance between any pair of vertices is the {\it diameter} $k$ of $G$. Given $i\le k$, the set of vertices at distance $i$ from vertex $u$ is denoted by $G_{i}(u)$.

As in the case of (undirected) graphs and digraphs, the degree/diameter problem for mixed graphs calls for finding the largest possible number of vertices $N(r,z,k)$ in a mixed graph with maximum undirected degree $r$, maximum directed outdegree $z$, and diameter $k$.
A bound for $N(r,z,k)$ is called a Moore(-like) bound. It is obtained by counting the number of vertices of a {\em Moore tree\/} $MT(u)$ rooted at a given vertex $u$, with depth equal to the diameter $k$, and assuming that for any vertex $v$ there exists a unique shortest path of length at most $k$ (with the usual meaning when we see $G$ as a digraph) from $u$ to $v$. The number of vertices in $MT(u)$, which is denoted by $M(r,z,k)$, was given by Buset, Amiri, Erskine, Miller, and P\'erez-Ros\'es \cite{baemp15}, and it is the following:
\begin{equation}
\label{eq:moorebound4}
M(r,z,k)=A\frac{u_1^{k+1}-1}{u_1-1}+B\frac{u_2^{k+1}-1}{u_2-1},
\end{equation}
where
\begin{align*}
v   &=(z+r)^2+2(z-r)+1, \\
u_1 &=\displaystyle{\frac{z+r-1-\sqrt{v}}{2}}, \qquad
u_2 =\displaystyle{\frac{z+r-1+\sqrt{v}}{2}}, \\
A   &=\displaystyle{\frac{\sqrt{v}-(z+r+1)}{2\sqrt{v}}}, \qquad
B   =\displaystyle{\frac{\sqrt{v}+(z+r+1)}{2\sqrt{v}}}.
\end{align*}

This bound applies when $G$ is totally regular with degrees $(r,z)$. Moreover, if we bound the total degree $d=r+z$, the largest number is always obtained when $r=0$ and $z=d$. That is, when the mixed graph has no (undirected) edges. In Table~\ref{table:1} we show the values of \eqref{eq:moorebound4} when $r=d-z$, with $0\le z\le d$, for different values of $d$ and diameter $k$. In particular, when $z=0$, the bound corresponds to the Moore bound for graphs (numbers in bold).

\medskip

\begin{table}[h!]
\centering
\begin{tabular}{||@{\,}ccccccc@{\,}||}
\hline\hline
$d\diagdown k$  & \vline & $1$ & $2$ & $3$ & $4$ & $5$ \\
\hline
1& \vline  &   {\bf 2}   &     $z+{\bf 2}$    &   $2z+{\bf 2}$ &  $z^2+2z+{\bf 2}$   &  $2z^2+2z+{\bf 2}$ \\
2& \vline & {\bf 3} & $z+{\bf 5}$ &  $4z+{\bf 7}$ & $z^2+9z+{\bf 9}$ & $5z^2+16z+{\bf 11}$ \\
3& \vline & {\bf  4} & $z+{\bf 10}$ & $6z+{\bf 22}$ & $\z^2+22z+{\bf 46}$ & $8z^2+66z+{\bf 94}$ \\
4& \vline & {\bf 5} & $z+{\bf 17}$ & $8z+{\bf 53}$ & $z^2+41z+{\bf 161}$ & $11z^2+176z+{\bf 485}$ \\
5& \vline & {\bf 6} & $z+{\bf 26}$ & $10z+{\bf 106}$ & $z^2+66z+{\bf 426}$ & $14z^2+370z+{\bf 1706}$\\
\hline\hline
\end{tabular}
\caption{Moore bounds according to \eqref{eq:moorebound4}.}
\label{table:1}
\end{table}

\section{A new upper bound}

An alternative approach
for computing the bound given by \eqref{eq:moorebound4} is the following (see also  \cite{dfl16}).
Let $G$ be a $(r,z)$-regular mixed graph with $d=r+z$. Given a vertex $v$ and for $i=0,1,\ldots,k$, let $N_i=R_i+Z_i$ be the maximum possible number of vertices at distance
$i$ from $v$. Here, $R_i$ is the number of vertices that, in
the corresponding tree rooted at $v$, are adjacent by an edge to their parents; and  $Z_i$
is the number of vertices that are adjacent by an arc from their parents. Then,
\begin{equation}
\label{Ni}
N_i = R_i+Z_i = R_{i-1}((r-1)+z)+Z_{i-1}(r+z).
\end{equation}
That is,
\begin{align}
R_i & = R_{i-1}(r-1)+Z_{i-1}r,  \label{Ri}\\
Z_i & = R_{i-1}z+Z_{i-1}z, \label{Zi}
\end{align}
or, in matrix form,
$$
\left(
\begin{array}{c}
  R_i \\
  Z_i
\end{array}
\right)=
\left(
\begin{array}{cc}
  r-1 & r\\
  z   & z
\end{array}
\right)
\left(
\begin{array}{c}
  R_{i-1} \\
  Z_{i-1}
\end{array}
\right)=\cdots=\M^i\left(
\begin{array}{c}
  R_{0} \\
  Z_{0}
\end{array}
\right)=
\M^i\left(
\begin{array}{c}
  0 \\
  1
\end{array}
\right),
$$
where $\M=\left(
\begin{array}{cc}
  r-1 & r\\
  z   & z
\end{array}
\right)$ and, by convenience, $R_0=0$ and $Z_0=1$. Therefore,
$$
N_i = R_i+Z_i =\left(
\begin{array}{cc}
  1 & 1
\end{array}
\right)\M^i\left(
\begin{array}{c}
  0 \\
  1
\end{array}
\right).
$$
Consequently, after summing a geometric matrix progression, the order of $MT(u)$ turns out to be
\begin{equation}
\label{MooreBound2}
  M(r,z,k) = \sum_{i=0}^{k}N_i =\frac{1}{r+2z-2}\left(
\begin{array}{cc}
  1 & 1
\end{array}
\right)(\M^{k+1}-\I)\left(
\begin{array}{c}
  r \\
  z
\end{array}
\right),
\end{equation}
with $r+2z\neq2$, that is, except for the cases $(r,z)=(0,1)$ and $(r,z)=(2,0)$, which correspond to a directed and undirected cycle, respectively.

Alternatively, note that $N_i$ satisfies an easy linear recurrence formula (see again
Buset, El Amiri, Erskine, Miller, and P\'erez-Ros\'es~\cite{baemp15}).
Indeed, from \eqref{Ni} and \eqref{Zi} we have that
$Z_i=z(N_{i-1}-Z_{i-1})+zZ_{i-1}=zN_{i-1}$ and, hence,
\begin{align}
\label{recurNi}
N_i & =(r+z)N_{i-1}-R_{i-1}=(r+z)N_{i-1}-(N_{i-1}-Z_{i-1}) \nonumber\\
    & =(r+z-1)N_{i-1}+zN_{i-2},\qquad i=2,3,\ldots
\end{align}
with initial values $N_0=1$ and $N_1=r+z$.

In this context,  Nguyen, Miller, and Gimbert \cite{nmg07} showed that the bound in \eqref{eq:moorebound4} is not attained for diameter $k\ge 3$ and, hence, that {\em mixed Moore graphs\/} do not exist in general.
More precisely, they proved that there exists a pair of vertices $u,v$ such that there are two different paths of length $\le k$ from $u$ to $v$. When there exist exactly two such paths, the usual terminology is to say that $v$ is the {\em repeat} of $u$, and this is denoted by writing $\rep(u)=v$ (see, for instance, Miller and \v{S}ir\'a\v{n} \cite{ms13}). Extending this concept, we denote by $\Rep(u)$ the set (or multiset)  of vertices $v$ such that there are $\nu\ge 2$ paths of length $\le k$ from $u$ to $v$, in such a way that each $v$ appears $\nu-1$ times in $\Rep(u)$. (In other words, we could say that vertex $v$ is ``repeated'' or ``revisited" $\nu-1$ times when reached from $u$.) Then, as a consequence, the number $N$ of vertices of $G$ must satisfy the bound
$$
N\le |MT(u)|-|\Rep(u)|= M(r,z,k)-|\Rep(u)|.
$$
We use this simple idea in the proof of our main result.

\begin{theorem}\label{the:newbound}
\label{theo1}
The order $N$ of a $(r,z)$-regular mixed graph $G$ with diameter $k\ge3$ satisfies the bound
\begin{equation}
\label{improved bound}
N\le M(r,z,k)-r,
\end{equation}
where $M(r,z,k)$ is given by \eqref{eq:moorebound4}.
\end{theorem}

\begin{proof}
It is clear that we can assume that there are no parallel arcs or edges. Let $u$ be a vertex with edges to the vertices $v_1,\ldots,v_r$ and arcs to the vertices $u_1,\ldots,u_z$. For each $i=1,\ldots,r$, let $v_{i1},\ldots,v_{iz}$ be the vertices adjacent (through arcs) from $v_i$. (The situation in the case $r=z=2$ is depicted in Figure~\ref{fig1}, where the dashed lines represent paths.) Now, for some fixed $i=1,\ldots,r$
and $j=1,\ldots,z$,
let us consider the following possible  cases for the distance from a vertex in $\{u_1,\ldots,u_z\}$ to vertex $v_{ij}$:

\begin{figure}[t]
    \begin{center}
        \includegraphics[width=12cm]{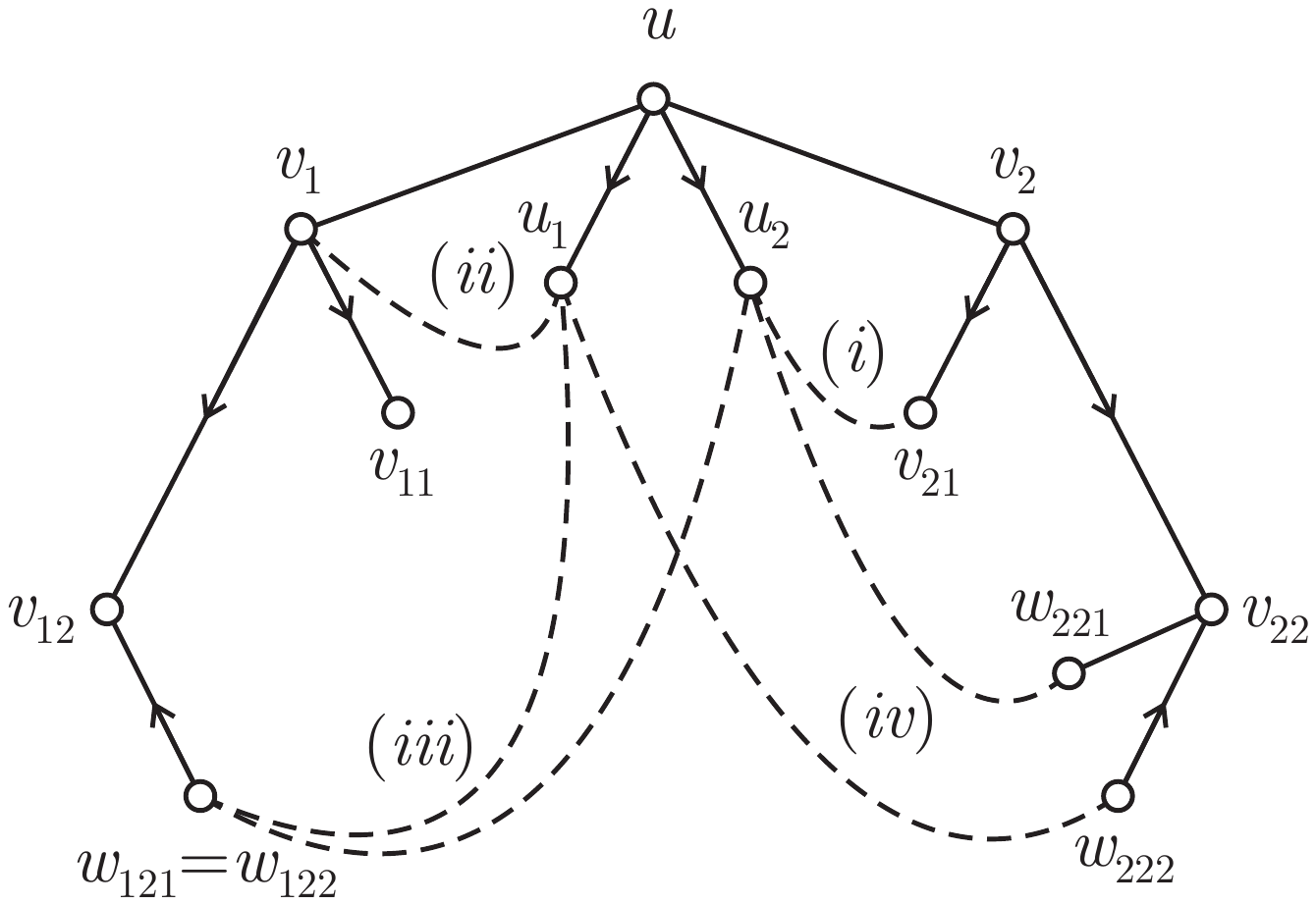}
    \end{center}
    \vskip-11.25cm
	\caption{Repeated vertices in a $(2,2)$-regular mixed graph: $(i)$ $v_{21}\in \Rep(u)$;
$(ii)$ $v_{1}\in \Rep(u)$;
$(iii)$ $w_{121}\in \Rep(u)$;
$(iv)$ $w_{221}\in \Rep(u)$.}
	\label{fig1}
\end{figure}

\begin{itemize}
\item[$(i)$]
If, for some $h=1,\ldots,z$, we have $\dist(u_h,v_{ij})<k$, then there exist two paths of length at most $k$ from $u$ to $v_{ij}$ and, hence, $v_{ij}\in \Rep(u)$ (note that this includes the case $u_h=v_{ij}$).
\item[$(ii)$]
If, for some $h=1,\ldots,z$, we have $\dist(u_h,v_{ij})=k$ and the shortest path from $u_h$ to $v_{ij}$ goes through $v_i$, then there are two paths of length $\le k$ from $u$ to  $v_{i}$ (one of length 1 and the other of length $k$). Hence, $v_{i}\in \Rep(u)$. In fact, notice that, in this case, $\dist(u_h,v_{i\ell})=k$ for every $\ell=1,\ldots,z$.
\end{itemize}
If, for every $h=1,\ldots,z$, we have $\dist(u_h,v_{ij})=k$, let $w_{ijl}$ denote, for $\ell=1,\ldots,z$, the predecessor vertices to $v_{ij}$ in the paths (of length $k$) from every $u_h$ to $v_{ij}$ (see the dashed lines in Figure~\ref{fig1}). Now we have again two cases:
\begin{itemize}
\item[$(iii)$]
If, for some $\ell,\ell'=1,\ldots,z$, we have $w_{ij\ell}=w_{ij\ell'}$, then there are two paths of length $k$ from $u$ to  $w_{ij\ell}$. Thus, $w_{ij\ell}\in \Rep(u)$.
\item[$(iv)$]
Otherwise, since $z^-(v_{ij})=z$, there must be at least one $\ell$ such that $w_{ij\ell}v_{ij}$ is an edge. But, in this case, there are two paths from $u$ to $w_{ij\ell}$ of length at most $k(\ge 3)$ and, so, $w_{ij\ell}\in \Rep(u)$.
\end{itemize}
As a consequence, we see that, for each $i=1,\ldots,r$  there is a vertex, which is either $v_i$, $v_{ij}$, or $w_{ij\ell}$, belonging to $\Rep(u)$.
Moreover, different values of $i$ lead to different repeated vertices, so that the paths from $u$ to them must be also different. In any case, the multiset $\Rep(u)$ has at least $r$ elements, and the result follows.
\end{proof}

The new upper bound $M(r,z,k)-r$ for diameter $k \geq 3$ can be even improved for certain cases, as the next proposition states.
\begin{proposition}
\label{propo}
Let $G$ be a $(r,z)$-regular mixed graph of diameter $k \geq 3$ with order $N$. If $r$ and $z$ are odd, and $k\equiv 2\ \mod 3$, then
 \begin{equation}
N \le M(r,z,k)-r-1.
\end{equation}
\end{proposition}
\begin{proof}
The proof is based on a parity argument. Namely, since $r$ is odd, $N$ must be even.
Thus, let us check the parity of $M(r,z,k)-r=\sum_{i=0}^k N_i-r$.
Let $\pi_i\in\{0,1\}$ denote the parity of $N_i$ in the obvious way.
If $z$ is odd, we have that $\pi_0=1$, $\pi_1=0$ and, from \eqref{recurNi} we get
the recurrence $\pi_i=\pi_{i-1}+\pi_{i-2}$ for $i\ge 2$. This gives the following sequence for the $\pi_i$'s: $1,0,1,1,0,1,1,0,1,1,\ldots$ Thus, $\sum_{i=0}^k N_i$ is even for every $k\equiv 2\ \mod 3$. Then, as $r$ is odd, we get the result.
\end{proof}

\section{The case of $(1,1)$-regular mixed graphs with diameter three}

In this section we show that the upper bound \eqref{improved bound} is attained for exactly three mixed graphs in the case $r=z=1$ and $k=3$.

\begin{figure}[t]
    \begin{center}
        \includegraphics[width=14cm]{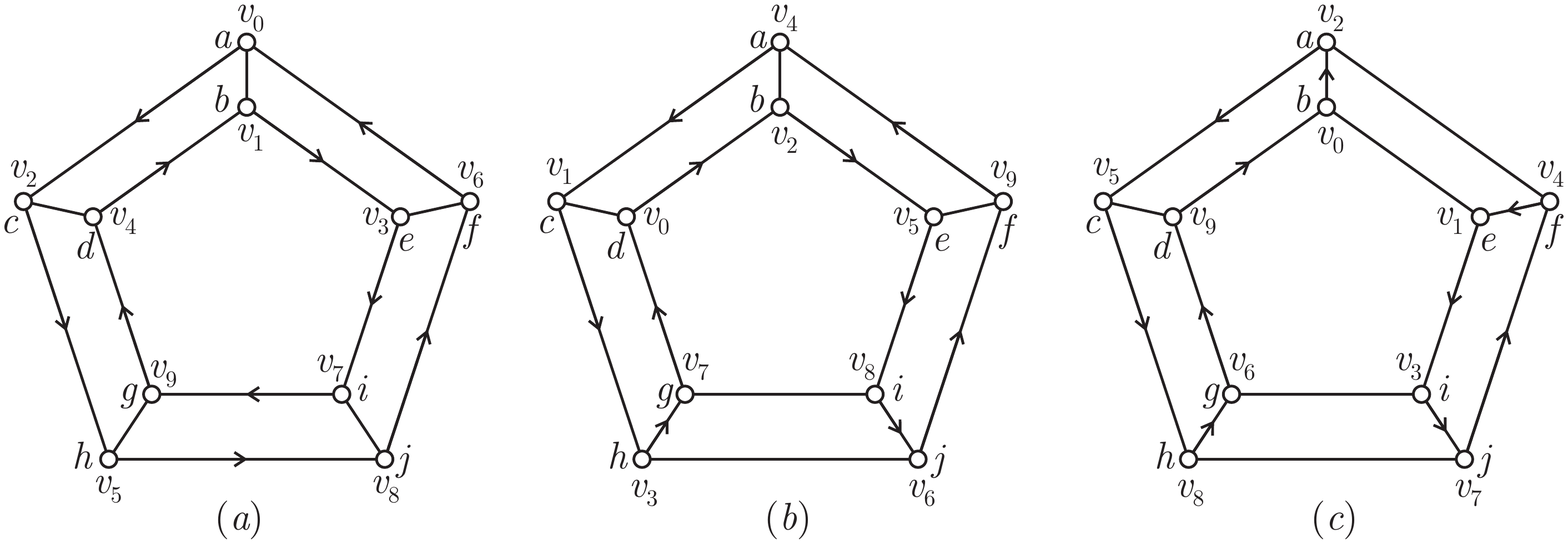}
    \end{center}
    \vskip-5.5cm
	\caption{The unique three non-isomorphic $(1,1)$-regular mixed graphs with diameter $k=3$ and order $N=10$.}
\label{fig2}
\end{figure}

\begin{proposition}
\label{propo3graphs}
Let $G$ be a $(1,1)$-regular mixed graph with diameter $k=3$ and maximum order $N=10$ given by \eqref{improved bound}. Then, $G$ is isomorphic to one of the three mixed graphs depicted in Figure~\ref{fig2}.
\end{proposition}
\begin{proof}
We divide the proof according to the four cases $(i)$--$(iv)$ given in Theorem \ref{the:newbound}. Let $u$ be any vertex of $G$. The remaining vertices of $G$ fall into one of the sets $G_i(u)$, according to their corresponding distance  $i \in \{1,2,3\}$ from $u$. Then, $|G_1(u)|=2$, and it is easy to see that $|G_2(u)|=3$ and $|G_3(u)|=4$ since, otherwise, $G$ would have order $N<M(1,1,3)-1=10$. Now, observe that case $(i)$ is impossible since $\dist(u_1,v_{11})<3$ would imply $|G_3(u)|<4$. Also, case $(iii)$ is not possible simply because $z=1$. So, let us suppose that we are in case $(ii)$, that is, $\dist(u_1,v_{11})=3$ and the shortest path from $u_1$ to $v_{11}$ goes through $v_1$. Hence, $G$ contains one of the two induced mixed subgraphs depicted in Figure~\ref{fig3} (from now on, we follow the vertex labeling in this figure, where $v_0=u, v_2=u_1$ and $v_3=v_{11}$). Next, we proceed in detail with case $(iia)$ and we leave to the reader cases $(iib)$ and $(iv)$, where similar reasoning leads to the same mixed graphs.

\begin{figure}[t]
    \begin{center}
\includegraphics[width=12cm]{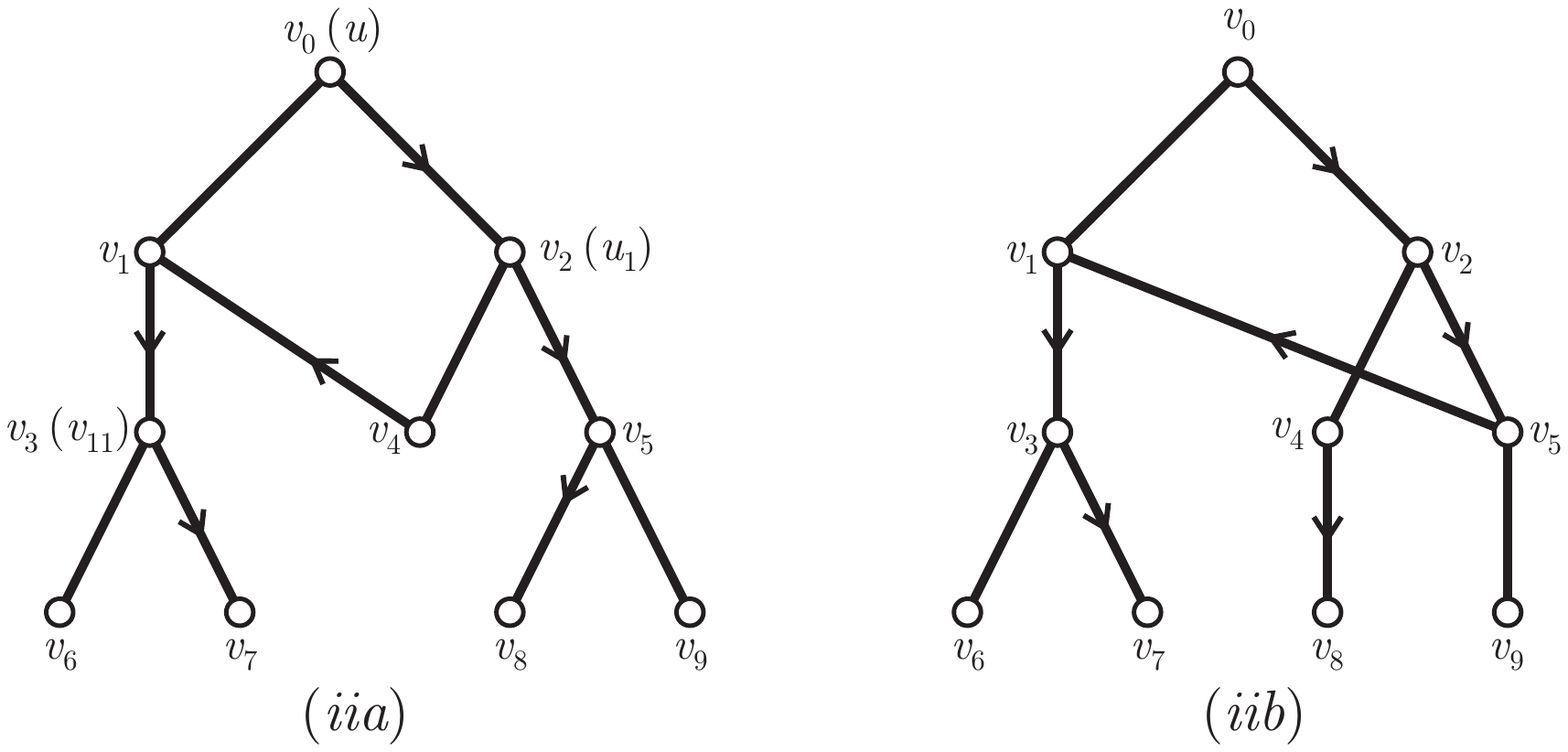}
\end{center}
    \vskip-12.25cm
	\caption{The two cases derived from $(ii)$ according to Theorem \ref{the:newbound} when $r=1,z=1$ and $k=3$.}
\label{fig3}
\end{figure}

\begin{figure}[t]
    \begin{center}
\includegraphics[width=16cm]{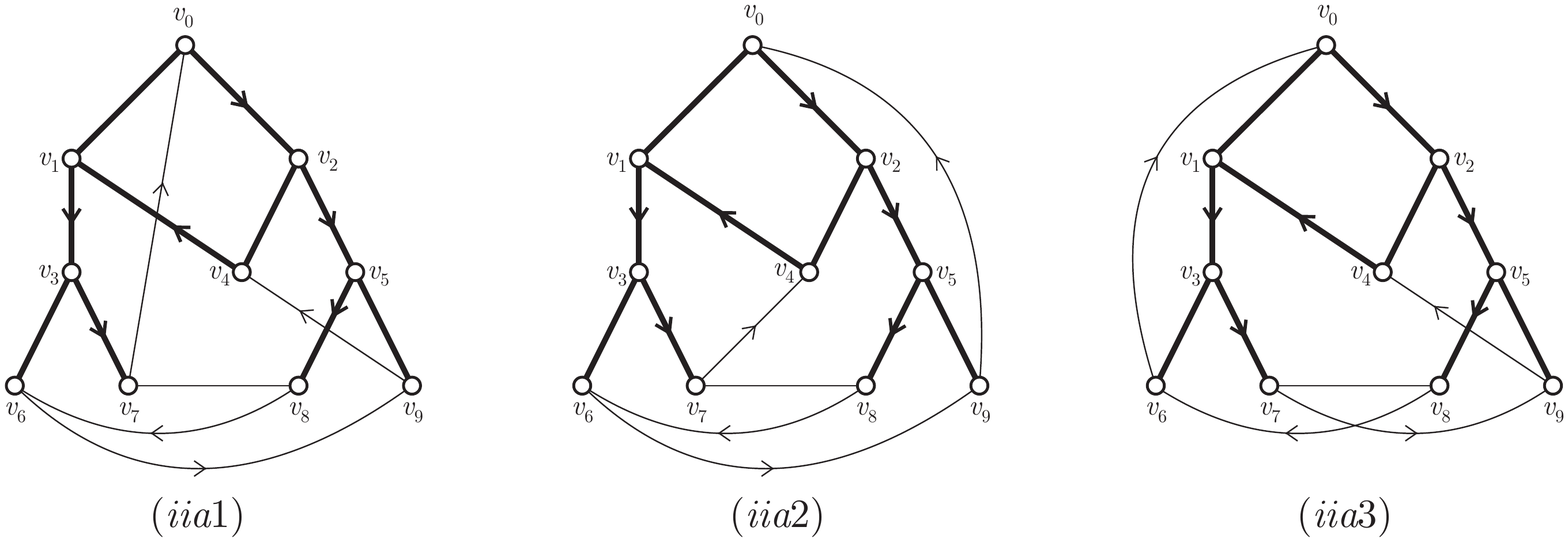}
\end{center}
    \vskip-6.25cm
	\caption{Three cases derived from $(iia)$ giving non-isomorphic mixed graphs.}
\label{figMA2}
\end{figure}

Due to its regularity, $G$ must contain the edge $v_7 \sim v_8$. Moreover, every vertex of $G$ is at distance $\leq 3$ from $v_2$ except $v_6$. This means that there must exist an arc $x\rightarrow v_6$, where $x \in \{v_8,v_9\}$.

\begin{itemize}
\item[$\bullet$]
Let $x=v_8$. Another arc $y\rightarrow v_9$ is needed to have $\dist(v_1,v_9) \leq 3$, where $y \in \{v_6,v_7\}$.
\begin{itemize}
\item[$\bullet$]
If $y=v_6\rightarrow v_9$ we have just two possibilities to complete the regularity of the mixed graph:
\begin{itemize}
\item[$\bullet$]
The remaining arcs are $v_7\rightarrow v_0$ and $v_9\rightarrow v_4$, which yield the mixed graph of Figure~\ref{figMA2}$(iia1)$, which is isomorphic to the one in Figure~\ref{fig2}$(b)$.
\item[$\bullet$]
The last arcs are  $v_7\rightarrow v_4$ and $v_9\rightarrow v_0$, in which case we obtain the mixed graph of Figure~\ref{figMA2}$(iia2)$, which is isomorphic to the one in Figure~\ref{fig2}$(c)$.
\end{itemize}
\item[$\bullet$]
If $y=v_7\rightarrow v_9$, we have again two possibilities:
\begin{itemize}
\item[$\bullet$]
The arcs $v_6\rightarrow v_0$ and $v_9\rightarrow v_4$ yield the mixed graph of Figure~\ref{figMA2}$(iia3)$, which is isomorphic to the one in Figure~\ref{fig2}$(a)$.
\item[$\bullet$]
The arcs $v_6\rightarrow v_4$ and $v_9\rightarrow v_0$ give rise to a mixed graph isomorphic to the one in Figure~\ref{fig2}$(b)$.
\end{itemize}
\end{itemize}
A scheme of the above cases is the following.
$$
x=v_8\rightarrow v_6\ \Rightarrow\
\left\{
\begin{array}{l}
y=v_6\rightarrow v_9\ \Rightarrow\
\left\{
\begin{array}{l}
v_7\rightarrow v_0\ \&\ v_9\rightarrow v_4\ \rightsquigarrow \ (b)\\
\mbox{or} \\
v_7\rightarrow v_4\ \&\ v_9\rightarrow v_0\ \rightsquigarrow \ (c)
\end{array}
\right.
\\
\mbox{or}
\\
y=v_7\rightarrow v_9\ \Rightarrow\
\left\{
\begin{array}{l}
v_6\rightarrow v_0\ \&\ v_9\rightarrow v_4\ \rightsquigarrow \ (a)\\
\mbox{or} \\
v_6\rightarrow v_4\ \&\ v_9\rightarrow v_0\ \rightsquigarrow \ (b)
\end{array}
\right.
\end{array}
\right.
$$
\item[$\bullet$]
Let $x=v_9$. We must add the arc $v_7\rightarrow v_9$ in order to have $\dist(v_1,v_9) \leq 3$. Now, to complete the mixed graph we have two possibilities:
\begin{itemize}
\item[$\bullet$]
The arcs $v_6\rightarrow v_0$ and $v_8\rightarrow v_4$ yield a mixed graph isomorphic to the one in Figure~\ref{fig2}$(b)$.
\item[$\bullet$]
The arcs $v_6\rightarrow v_4$ and $v_8\rightarrow v_0$ complete a mixed graph isomorphic to the one in Figure~\ref{fig2}$(c)$.
\end{itemize}
Schematically,
$$
x=v_9\rightarrow v_6\ \Rightarrow\ v_7\rightarrow v_9\ \Rightarrow\
\left\{
\begin{array}{l}
v_6\rightarrow v_0\ \&\ v_8\rightarrow v_4\ \rightsquigarrow \ (b)\\
\mbox{or}\\
v_6\rightarrow v_4\ \&\ v_8\rightarrow v_0\ \rightsquigarrow \ (c)
\end{array}
\right.
$$
\end{itemize}
This completes the proof.
\end{proof}


Note that the mixed graph  in Figure~\ref{fig2}$(a)$ is the line digraph of the cycle $C_5$ $($seen as a digraph, so that each edge corresponds to a digon\/$)$. It is also the Cayley graph of the dihedral group $D_5=\langle r,s\, |\, r^5\!=\!s^2\!=\!(rs)^2\!=\!1 \rangle$, with generators $r$ and $s$.
The spectrum of this mixed graph is that of the $C_5$ cycle plus a $0$ with multiplicity $5$.  Namely,
$$
\textstyle
\spec G=\left\{2,\ \left(-\frac{1}{2}+\frac{\sqrt{5}}{2}\right)^2,\ 0^5,\ \left(-\frac{1}{2}-\frac{\sqrt{5}}{2}\right)^2 \right\}.
$$
This is because $G$ is the line digraph of $C_5$. As a consequence, the only difference between $\spec G$ and $\spec C_5$ are the additional $0$'s (see Balbuena, Ferrero, Marcote, and Pelayo \cite{bfmp03}.)
In fact, the mixed graphs of Figures~\ref{fig2}$(b)$ and \ref{fig2}$(c)$ are cospectral with $G$, and can be obtained by applying  a recent method
to obtain cospectral digraphs with  a locally line digraph. The right modifications to obtain the mixed graphs $(b)$ and $(c)$ from mixed graph $(a)$ are depicted in Figure~\ref{metode-digraf}.
For more details, see Dalf\'o and Fiol \cite{df16}.

\begin{figure}[t]
    \begin{center}
\includegraphics[width=12cm]{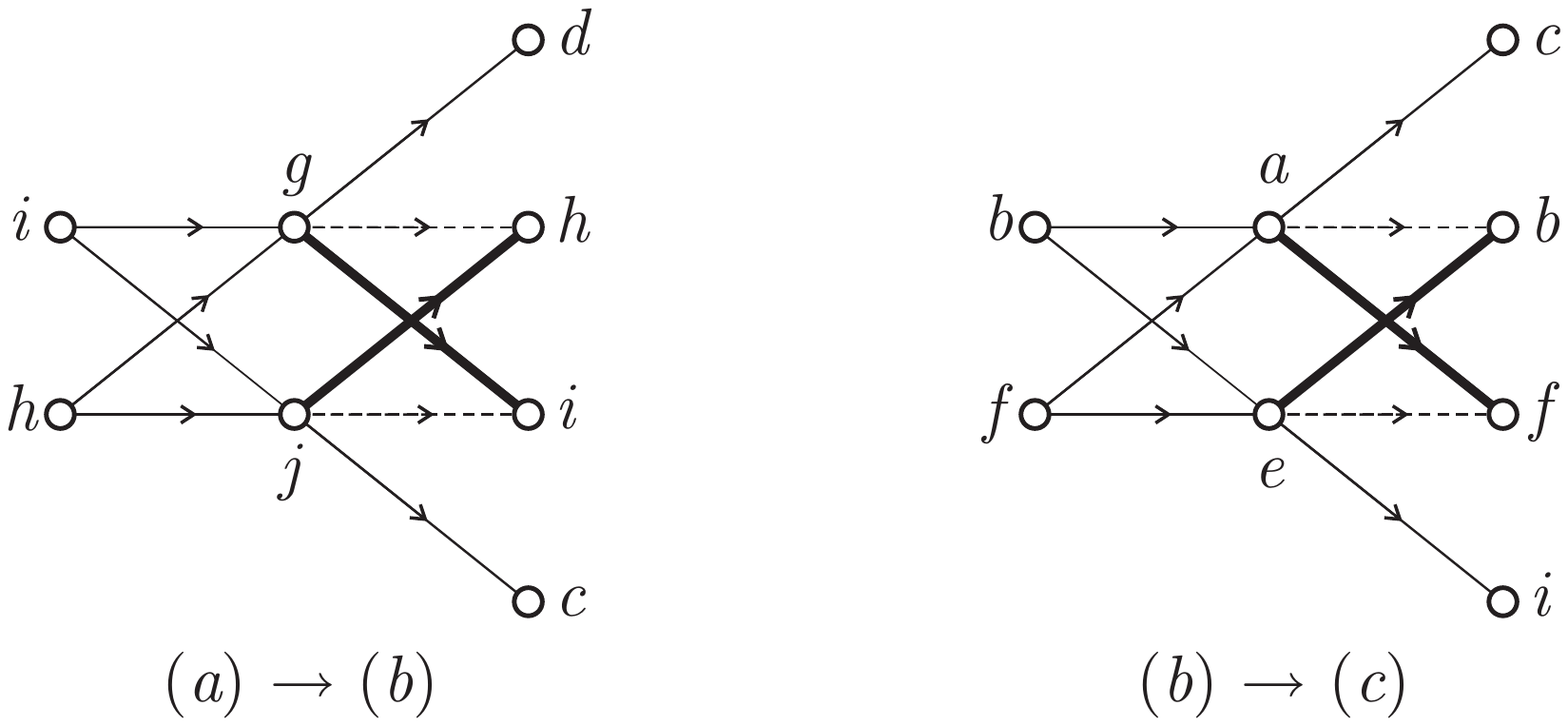}
\end{center}
    \vskip-12.5cm
	\caption{The method for obtaining the cospectral digraphs of Figure \ref{fig2}.}
\label{metode-digraf}
\end{figure}

Two other interesting characteristics of these mixed graphs are the following:
\begin{itemize}
\item
Each of the three mixed graphs is isomorphic to its converse (where the directions of the arcs are reversed).
\item
Each of these mixed graphs can be obtained as a proper orientation of the so-called Yutsis graph of the $15j$ symbol of the second kind (see Yutsis, Levinson, and Vanagas~\cite{ylv62}). This is also called the pentagonal prim graph. Notice that it has girth 4 and, curiously, its diameter is 3, in every of its considered orientations here.
\end{itemize}

The result of Proposition \ref{propo3graphs} could prompt us to look for a whole family of  $(1,1)$-regular mixed graphs attaining the upper bound $M(1,1,k)-1$ for any diameter $k \geq 3$. Nevertheless, as a consequence of Proposition \ref{propo}, this is not possible, since such a bound cannot be attained for some values of $k$.

\begin{corollary}
Let $G$ be a $(1,1)$-regular mixed graph with  $N$ vertices and diameter $k=2+3s$ with $s \geq 1$. Then,
\begin{equation}
\label{coro-fibo}
N \leq \theta_1\phi_1^{k+1}+\theta_2\phi_2^{k+1}-4,
\end{equation}
where $\theta_{1,2}=1\pm \frac{2}{\sqrt{5}}$ and $\phi_{1,2}=\frac{1}{2}(1\pm\sqrt{5})$.
\end{corollary}
\begin{proof}
Apply Proposition \ref{propo} with $r=z=1$ and $M(1,1,k)$ computed from \eqref{eq:moorebound4}.
\end{proof}
Note that, in this last case, \eqref{recurNi} yields the recurrence $N_i=N_{i-1}+N_{i-2}$, with $N_0=1$ and $N_1$, so defining a Fibonacci sequence.
In fact, with the usual numbering of such a sequence ($F_1=1$, $F_2=1$, $F_3=2$,\ldots), we have $M(1,1,k)=F_{k+4}-2$ and so, for the case under consideration, \eqref{coro-fibo} becomes
$$
N\le F_{k+4}-4.
$$

\vskip .6cm

\noindent{\bf Acknowledgments.}
The authors would like to thank an anonymous referee whose useful comments lead to a significant  improvement of the manuscript.
This research was partially supported by MINECO under project MTM2017-88867-P, and AGAUR under project 2017SGR1087 (C. D. and M. A. F.). The author N. L. has been supported, in part, by grant MTM2013-46949-P from the {\em Ministerio de Econom\'{\i}a y Competitividad}, and 2014SGR1666 {\em Catalan Research Council}.

%


\section*{References}
\begin{thebibliography}{99}

\bibitem{bfmp03}
C. Balbuena, D. Ferrero, X. Marcote, and I. Pelayo, Algebraic properties of a digraph and its line digraph,
{\em J. Interconnection Networks} {\bf 04} (2003), no. 4, 377--393.

\bibitem{baemp15}
D. Buset, M. El Amiri, G. Erskine, M. Miller, and H. P\'erez-Ros\'es,
A revised Moore bound for partially directed graphs, {\em Discrete Math.} {\bf 339} (2016), no. 8, 2066--2069.

\bibitem{df16}
C. Dalfó and M.A. Fiol, Cospectral digraphs from locally line digraphs,
{\em Linear Algebra Appl.} {\bf 500} (2016) 52--62.

\bibitem{dfl16}
C. Dalf\'o, M. A. Fiol, and N. L\'opez,
Sequence mixed graphs, {\em Discrete Appl. Math.} {\bf 219} (2017) 110--116.

\bibitem{lm16}
N. L\'opez and J. M. Miret, On mixed almost Moore graphs of diameter two, {\em Electron. J. Combin.} {\bf 23(2)} (2016) \#P2.3.

\bibitem{ms13}
M. Miller and J. \v{S}ir\'a\v{n},
Moore graphs and beyond: A survey of the degree/diameter problem, {\em Electron.
J.
Combin.} {\bf 20(2)} (2013) \#DS14v2.

\bibitem{nm08}
M. H. Nguyen and M. Miller, Moore bound for mixed networks, {\em Discrete Math.}
{\bf 308} (2008), no. 23, 5499--5503.

\bibitem{nmg07}
M. H. Nguyen, M. Miller, and J. Gimbert, On mixed Moore graphs, {\em Discrete Math.}
{\bf 307} (2007) 964--970.

\bibitem{ylv62}
A. P. Yutsis, L. B. Levinson, and V. V. Vanagas,
{\em Mathematical Apparatus of the Theory of Angular Momentum,}
Israel Program for Sci. Transl. Ltd., Jerusalem, 1962.

\end{thebibliography}
\end{document}